\documentclass[fleqn,reqno,11pt,a4paper,final]{amsart}
\usepackage[a4paper,left=40mm,right=40mm,top=40mm,bottom=40mm,marginpar=25mm]{geometry}

\usepackage{amsmath}
\usepackage{amssymb}
\usepackage{amsthm}
\usepackage{amscd}
\usepackage[ansinew]{inputenc}
\usepackage{bbm}
\usepackage{color}
\usepackage[english=american]{csquotes}
\usepackage{hyperref}
\usepackage{calc}
\usepackage{bm}
\usepackage{amsmath}
\usepackage{amssymb}
\usepackage{amsthm}
\usepackage{amscd}
\usepackage[ansinew]{inputenc}
\usepackage{bbm}
\usepackage{color}
\usepackage[english=american]{csquotes}
\usepackage{calc}
\usepackage{bm}
\usepackage{enumerate}
\usepackage{stmaryrd}

\usepackage{enumerate}

\numberwithin{equation}{section}

\newtheoremstyle{thmlemcorr}{10pt}{10pt}{\itshape}{}{\bfseries}{.}{10pt}{{\thmname{#1}\thmnumber{ #2}\thmnote{ (#3)}}}
\newtheoremstyle{thmlemcorr*}{10pt}{10pt}{\itshape}{}{\bfseries}{.}\newline{{\thmname{#1}\thmnumber{ #2}\thmnote{ (#3)}}}
\newtheoremstyle{defi}{10pt}{10pt}{\itshape}{}{\bfseries}{.}{10pt}{{\thmname{#1}\thmnumber{ #2}\thmnote{ (#3)}}}
\newtheoremstyle{remexample}{10pt}{10pt}{}{}{\bfseries}{.}{10pt}{{\thmname{#1}\thmnumber{ #2}\thmnote{ (#3)}}}
\newtheoremstyle{ass}{10pt}{10pt}{}{}{\bfseries}{.}{10pt}{{\thmname{#1}\thmnumber{ A#2}\thmnote{ (#3)}}}

\theoremstyle{thmlemcorr}
\newtheorem{theorem}{Theorem}
\numberwithin{theorem}{section}
\newtheorem{lemma}[theorem]{Lemma}

\theoremstyle{thmlemcorr*}
\newtheorem{theorem*}{Theorem}
\newtheorem{lemma*}[theorem]{Lemma}
\newtheorem{corollary*}[theorem]{Corollary}
\newtheorem{proposition*}[theorem]{Proposition}
\newtheorem{problem*}[theorem]{Problem}
\newtheorem{conjecture*}[theorem]{Conjecture}

\theoremstyle{defi}
\newtheorem{definition}[theorem]{Definition}

\theoremstyle{remexample}
\newtheorem{remark}[theorem]{Remark}

\theoremstyle{ass}

\newcommand{\Lrm}{\mathrm{L}}

\newcommand{\Hcal}{\mathcal{H}}

\newcommand{\Lcal}{\mathcal{L}}
\newcommand{\Mcal}{\mathcal{M}}

\newcommand{\Pcal}{\mathcal{P}}

\newcommand{\Nbf}{\mathbf{N}}

\DeclareMathOperator{\dom}{Dom}
\DeclareMathOperator{\im}{Im}

\DeclareMathOperator{\spn}{span}

\newcommand{\setn}[2]{\{\, #1 \ \ \textup{\textbf{:}}\ \ #2 \,\}}
\newcommand{\setb}[2]{\bigl\{\, #1 \ \ \textup{\textbf{:}}\ \ #2 \,\bigr\}}

\newcommand{\abs}[1]{|#1|}

\newcommand{\dd}{\;\mathrm{d}}

\newcommand{\N}{\mathbb{N}}
\newcommand{\R}{\mathbb{R}}

\DeclareMathOperator{\Tan}{Tan}

\newcommand{\term}[1]{\emph{#1}}

\newcommand{\proofstep}[1]{\textit{#1}}

\def\XXint#1#2#3{{\setbox0=\hbox{$#1{#2#3}{\int}$} 
\vcenter{\hbox{$#2#3$}}\kern-.5\wd0}}

\newcommand{\restrict}{\begin{picture}(10,8)\put(2,0){\line(0,1){7}}\put(1.8,0){\line(1,0){7}}\end{picture}}

\renewcommand{\phi}{\varphi}

\begin{document}

\title[On a conjecture of Cheeger]{On a conjecture of Cheeger}

\author[G.~De~Philippis]{Guido De Philippis}
\address{\textit{G.~De Philippis:} SISSA, Via Bonomea 265, 34136 Trieste, Italy.}
\email{guido.dephilippis@sissa.it}

\author[A.~Marchese]{Andrea Marchese}
\address{\textit{A.~Marchese:} Universit\"{a}t Z\"{u}rich, Winterthurerstrasse 190, CH-8057 Z\"{u}rich, Switzerland.}
\email{andrea.marchese@math.uzh.ch}

\author[F.~Rindler]{Filip Rindler}
\address{\textit{F.~Rindler:} Mathematics Institute, University of Warwick, Coventry CV4 7AL, UK.}
\email{F.Rindler@warwick.ac.uk}

\maketitle

\begin{abstract}
This note details how a recent structure theorem for normal $1$-currents proved by the first and third author  allows to prove a conjecture of Cheeger concerning the structure of Lipschitz differentiability spaces.  More precisely, we show that the  push-forward of the measure from a Lipschitz differentiability space under a chart is absolutely continuous with respect to Lebesgue measure.

\vspace{4pt}

\noindent\textsc{Keywords:} Lipschitz differentiability space, Cheeger's conjecture, Alberti representation, metric measure space.
\vspace{4pt}

\noindent\textsc{Date:} \today{}. 
\end{abstract}

\section{Introduction}
In~\cite{Cheeger99}  Cheeger  proved that in every doubling  metric measure space $(X,\rho,\mu)$  satisfying a Poincar\'{e} inequality, Lipschitz functions are differentiable $\mu$-almost everywhere. More precisely, he showed  the existence of a family $\{(U_i,\phi_i)\}_{i \in \N}$ of  Borel charts (that is, \(U_i\subset X\) is a Borel set, $X=\bigcup_iU_i$ up to a $\mu$-negligible set, and  $\phi_i \colon X \to \R^{d(i)}$ is Lipschitz) such that for every Lipschitz map $f \colon X \to \R$ at $\mu$-almost every \(x_0\in U_i\) there exists a unique (co-)vector \(df(x_0)\in \R^{d(i)}\) with
\[
  \limsup_{x\to x_0} \frac{\abs{f(x) - f(x_0) - df(x_0) \cdot (\phi(x) - \phi(x_0))}}{\rho(x,x_0)} = 0.
\]
This fact was later axiomatized by Keith~\cite{Keith04}, leading to the notion of \term{Lipschitz differentiability space}, see Section~\ref{sc:setup} below. 

Cheeger also conjectured that the push-forward of the reference measure \(\mu\) under every chart \(\varphi_i\) has to be absolutely continuous with respect to the Lebesgue measure, that is,
\begin{equation*} \label{eq:Cheeger}
  (\phi_i)_\# (\mu \restrict U_i) \ll \Lcal^{d(i)}\,,
\end{equation*}
see~\cite[Conjecture 4.63]{Cheeger99}.
Some consequences of this fact concerning existence of bi-Lipschitz embeddings of \(X\) into some \(\R^N\) are detailed in~\cite[Section 14]{Cheeger99}, also see~\cite{CheegerKleiner06,CheegerKleiner09}

Let us assume that  \((X,\rho,\mu) = (\R^d,\rho_{\mathcal E}, \nu)\) with \( \rho_{\mathcal E}\) the Euclidean distance and \(\nu\) a positive Radon measure, is a Lipschitz differentiability space when equipped with the (single) identity chart (note that it follows a-posteriori from the validity of Cheeger's conjecture that no mapping into a higher-dimensional space can be a chart in a Lipschitz differentiability structure of $\R^d$). In this case the validity of Cheeger's conjecture reduces  to the validity of the  (weak) converse of Rademacher's theorem, which states that a  positive Radon measure \(\nu\) on \(\R^d\) with the property that all Lipschitz functions are differentiable \(\nu\)-almost everywhere must be absolutely continuous with respect to \(\Lcal^d\). Actually, it is well known to experts that this converse of Rademacher's theorem  implies  Cheeger's conjecture in any metric space, see for instance~\cite[Section 2.4]{Keith04},~\cite[Remark~6.11]{Bate15}, and~\cite{Gong11}.

The (strong) converse of Rademacher's theorem has been known to be true  in \(\R\)  since  the work of  Zahorski~\cite{Zahorski46}, where he characterized the sets $E\subset\R$ that are sets of non-differentiability points of some Lipschitz function. In particular, he proved that for every Lebesgue negligible set \(E\subset \R\) there exists a Lipschitz function which is nowhere differentiable on \(E\). 

The same  result for maps  \(f \colon \R^d\to \R^d\) has been proved by Alberti, Cs\"ornyei \& Preiss  for  \(d=2\) as a consequence of a deep structural result for negligible sets in the plane~\cite{AlbertiCsornyeiPreiss05,AlbertiCsornyeiPreiss10}. In 2011, Cs\"ornyei \& Jones~\cite{Jones11talk}
announced  the extension of the above result to every Euclidean space. For Lipschitz maps $f \colon \R^d \to \R^m$ with $m < d$ the situation is fundamentally different and there exists a null set such that every Lipschitz function is differentiable at at least one point from that set, see~\cite{Preiss90,PreissSpeight14}.
We finally remark that the weak converse of Rademacher's theorem in \(\R^2\) can also be obtained by combining the results of~\cite{Alberti93} and~\cite{AlbertiMarchese16}, see~\cite[Remark~6.2~(iv)]{AlbertiMarchese16}.

Recently, a result concerning the singular structure of measures satisfying a differential constraint was proved in~\cite{DePhilippisRindler16}. When combined with the main result of~\cite{AlbertiMarchese16}, this proves  the weak converse of Rademacher's theorem in any dimension, see~\cite[Theorem~1.14]{DePhilippisRindler16}.

In this note we detail how the results in~\cite{AlbertiMarchese16,DePhilippisRindler16} in conjunction with Bate's result on the existence of a sufficient number of independent Alberti representations in a Lipschitz differentiability space~\cite{Bate15} imply Cheeger's conjecture; see Section~\ref{sc:setup} for the relevant definitions.

\begin{theorem} \label{thm:main}
Let $(X,\rho,\mu)$ be a Lipschitz differentiability space and let \((U,\varphi)\) be a \(d\)-dimensional chart. Then, 
\(
\varphi_\# (\mu\restrict U)\ll \Lcal^d.
\)
\end{theorem}

Note that by the same arguments of this paper Cheeger's conjecture would also follow from the results announced in~\cite{AlbertiCsornyeiPreiss05} and~\cite{Jones11talk}. 

After we finished writing  this note we learned that  similar  results have been proved by Kell and Mondino~\cite{KellMondino16} and by Gigli and Pasqualetto~\cite{GigliPasqualetto16}.

\subsection*{Acknowledgments} 
The authors would like to thank Nicola Gigli and Andrea Schioppa  for suggesting to write this note and for several useful discussions. G.~D.~P. is supported by the MIUR SIR-grant ``Geometric Variational Problems'' (RBSI14RVEZ). A.~M. is supported by the ERC-grant ``Regularity of area-minimizing currents'' (306247).
F.~R.\ gratefully acknowledges the support from an EPSRC Research Fellowship on ``Singularities in Nonlinear PDEs'' (EP/L018934/1).

\section{Setup} \label{sc:setup}

\subsection{Lipschitz differentiability spaces}
In the sequel, the triple  $(X,\rho,\mu)$ will always denote a  \term{metric measure space}, that is, $(X,\rho)$ is a separable, complete metric space and $\mu \in \Mcal_+(X)$ is a positive  Radon measure on $X$. 

We call a pair \((U,\phi)\) such that  $U \subset X$ is a Borel set and  $\phi \colon X \to \R^d$ is Lipschitz, a \term{$d$-dimensional chart}, or simply a \term{$d$-chart}. A function  $f \colon X \to \R$ is said to be \term{differentiable with respect to a $d$-chart $(U,\phi)$}  at $x_0 \in U$ if there exists a unique (co-)vector $df(x_0) \in \R^d$ such that
\[
  \limsup_{x\to x_0} \frac{\abs{f(x) - f(x_0) - d f(x_0) \cdot (\phi(x) - \phi(x_0))}}{\rho(x,x_0)} = 0.
\]

We call a metric measure space $(X,\rho,\mu)$ a \term{Lipschitz differentiability space} (also called a metric measure space that admits a \term{measurable differentiable structure}) if there exists a countable family of \(d(i)\)-charts $(U_i,\phi_i)$ ($i \in \N$) such that $X=\bigcup_iU_i$ and any Lipschitz map $f \colon X \to \R$ is differentiable with respect to every $(U_i,\phi_i)$ at $\mu$-almost every point $x_0 \in U_i$.

\subsection{Alberti representations}\label{sc:s1}
We denote by  $\Gamma(X)$ the set of \term{curves} in $X$, that is, the set of all Lipschitz maps $\gamma \colon \dom \gamma \to X$, for which the domain $\dom \gamma \subset \R$ is non-empty and compact. Note that we are not requiring \(\dom \gamma\) to be  an interval and thus the set \(\Gamma(X)\) is sometimes also called the set of \term{curve fragments} on \(X\).   We equip $\Gamma(X)$ with the Hausdorff metric ${\rm{dist}}_{\Hcal}$ on graphs and we consider it as a subspace of the Polish space
\begin{equation}\label{eq:K}
\mathcal K=\setb{K\subset \R\times X}{ \text{$K$ compact}},
\end{equation}
endowed with the Hausdorff metric. Moreover, by arguing as in~\cite[Lemma 2.20]{Schioppa13}, it is easy to see that \(\Gamma(X)\) is an \(F_\sigma\)-subset of \(\mathcal K\), i.e.\ a countable union of closed sets.

The decomposition of a measure into a family of \(1\)-dimensional Hausdorff measures supported on curves leads to the notion of Alberti representation. First introduced in~\cite{Alberti93} for the study of the rank-one property of BV-derivatives, this decomposition has turned out to be a key tool in the study of differentiability properties of Lipschitz functions, see for instance~\cite{AlbertiCsornyeiPreiss05,AlbertiCsornyeiPreiss10, AlbertiMarchese16,Bate15}.

\begin{definition}\label{def:Albertirepresentation}
Let $(X,\rho,\mu)$ be a metric measure space. An \term{Alberti representation} of $\mu$ on a $\mu$-measurable set $A \subset X$
is a parametrized family $(\mu_\gamma)_{\gamma \in \Gamma(X)}$ of positive Borel measures $\mu_\gamma \in \Mcal_+(X)$ with
\[
  \mu_\gamma \ll \Hcal^1 \restrict \im \gamma,
\]
together with a Borel probability measure $\pi \in \Pcal(\Gamma(X))$ such that
\begin{equation}\label{eq:dec}
  \mu(B) = \int \mu_\gamma(B) \dd \pi(\gamma)  \qquad
  \text{for all Borel sets $B \subset A$.}
\end{equation}
Here, the measurability of the integrand is part of the requirement of being an Alberti representation
\end{definition}

\begin{remark}\label{rmk:bilip}
Note that this definition is slightly different from the one in~\cite[Definition 2.2]{Bate15} since there the set  \(\Gamma(X)\) consist of \term{bi-Lipschitz} curves. Clearly, the existence of a representation in the sense of~\cite{Bate15} implies the existence of a representation in our sense and this will suffice for our purposes. Let us, however, point out that the converse holds true as well. Indeed, the part of \(\gamma\) that contributes to the integral in~\eqref{eq:dec} can be decomposed into countably many  bi-Lipschitz pieces, see~\cite[Remark~2.17]{Schioppa13}. 
\end{remark}

We will further need the notion of \emph{independent} Alberti-representations of a measure. Let  $C\subset\R^d$ be a closed, convex, one-sided cone, i.e.\ a set of the form
\[
  C:=\setb{v\in\R^d}{v\cdot w\geq(1-\theta)\|v\|}
\]
for some $w\in\mathbb{S}^{d-1}$ and $\theta \in (0,1)$. With a Lipschitz map $\varphi \colon X \to \R^d$, we say that an Alberti representation $\int \nu_{\gamma} \dd \pi(\gamma)$ has \term{$\phi$-directions in $C$} if
\[
  (\phi \circ \gamma)'(t) \in C \setminus \{0\}  \qquad
  \text{for $\pi$-a.e.\ curve $\gamma$ and $\Hcal^1$-a.e.\ $t \in \dom \gamma$.}
\]
A number of $m$ Alberti representations of $\mu$ are \term{$\phi$-independent} if there are linearly independent cones $C_1,\ldots,C_m$ such that the $i$'th Alberti representation  has $\phi$-directions in $C_i$. Here, linear independence of the cones $C_1,\ldots,C_m$ means that any collection of vectors $v_i \in C_i \setminus \{0\}$ is linearly independent. In the case \(X=\R^d\) we will always consider \(\varphi={\rm Id}\).

One of the main results of~\cite{Bate15} asserts that a Lipschitz differentiability space necessarily admits many independent Alberti representations, also cf.~\cite[Theorem~1.1]{AlbertiMarchese16}. Recall that according to Remark~\ref{rmk:bilip} any representation in the sense of~\cite{Bate15} is also a representation in the sense of Definition~\ref{def:Albertirepresentation}.

\begin{theorem} \label{thm:Bate}
Let $(X,\rho,\mu)$ be a Lipschitz differentiability space with a $d$-chart $(U,\phi)$. Then, there exists a countable decomposition
\[
  U = \bigcup_{k \in \N} U_k,  \qquad \text{$U_k \subset U$ Borel sets,}
\]
such that every $\mu \restrict U_k$ has $d$ $\phi$-independent Alberti representations.
\end{theorem}

A proof of this theorem can be found in~\cite[Theorem~6.6]{Bate15}.

\subsection{One-dimensional currents}
In order to use the results of~\cite{DePhilippisRindler16} we need a link between Alberti representation and \(1\)-dimensional currents. Recall that  a  \term{$1$-dimensional current} $T$ in $\R^d$ is a continuous linear functional on the 
space of smooth and 
compactly supported differential $1$-forms on $\R^d$. 
The \term{boundary} of $T$,
$\partial T$ is the distribution ($0$-current) defined  via $\langle\partial T, f \rangle := \langle T, df\rangle $
for every smooth and compactly supported function $f \colon \R^d \to \R$.
The \term{mass} of $T$, denoted by 
$\mathbf{M}(T)$, is the supremum of $\langle T, \omega\rangle$ over
all $1$-forms $\omega$ such that $|\omega|\le 1$
everywhere. In particular, finite-mass currents can be naturally identified with \(\R^d\)-valued Radon measures. A current $T$ is called \term{normal} if both $T$ 
and $\partial T$ have finite mass; we denote the set of normal $1$-currents by $\mathbf{N}_1(\R^d)$.

By the Radon--Nikod\'{y}m theorem, a $1$-dimensional current $T$ with finite mass can be written in the form
$T=\vec{T} \|T\|$ where $\|T\|$ is a finite positive measure
and $\vec{T}$ is a vector field in $\Lrm^1(\R^d,\|T\|)$ with $|\vec{T}(x)|=1$ for $\|T\|$-almost every $x \in \R^d$.
In particular, the action of $T$ on a smooth and compactly supported $1$-form
$\omega$ is given by
\[
\langle T, \omega\rangle 
= \int_{\R^d} \langle\omega(x), \vec{T}(x)\rangle \dd\|T\|(x)
\; .
\]

An \term{integer-multiplicity rectifiable $1$-current} (in the following called simply rectifiable $1$-current) $T= \llbracket E,\tau,m \rrbracket$ is a $1$-current which acts on $1$-forms $\omega$ as
\[
\langle T, \omega\rangle 
= \int_{E} \langle\omega(x), \tau(x)\rangle \, m(x) \dd \Hcal^1(x)
\; ,
\]
where $E$ is a $1$-rectifiable set, $\tau(x)$ is a unit vector spanning the approximate tangent space $\Tan(E,x)$ and $m$ is an integer-valued function such that $\int_E m \dd \Hcal^1<\infty$. More information on currents can be found in~\cite{Federer69book}.

The relation between Alberti representations and normal $1$-currents is partially encoded in the following decomposition theorem, due to Smirnov~\cite{Smirnov93}.

\begin{theorem}
\label{s-decompcurr}
Let $T=\vec{T} \|T\| \in \mathbf{N}_1(\R^d)$ be a normal $1$-current
with $|\vec{T}(x)|=1$ for $\|T\|$-almost every~$x$.
Then, there exists a family of rectifiable $1$-currents
\[
  T_\gamma= \llbracket E_\gamma,\tau_\gamma,1 \rrbracket,  \qquad \gamma\in\Gamma,
\]
where $\Gamma$ is a measure space endowed with a finite positive Borel measure $\pi \in \Mcal_+(\Gamma)$, 
such that the following assertions hold:

\begin{enumerate}[(i)]
\item
$T$ can be decomposed as
\[
T=\int_\Gamma T_\gamma\dd \pi (\gamma)
\]
and
\[
\mathbf{M}(T) 
= \int_\Gamma \mathbf{M}(T_\gamma) \dd \pi(\gamma)
= \int_\Gamma \Hcal^1(E_\gamma) \dd \pi(\gamma)
  \; ; 
\]
\item
$\tau_\gamma(x)=\vec{T}(x)$ for $\Hcal^1$-almost every $x\in E_\gamma$ 
and for \(\pi\)-almost every $\gamma\in \Gamma$;
\item
$\|T\|$ can be decomposed as
\[
  \|T\|=\int_\Gamma  \mu_\gamma\dd \pi (\gamma) \; ,
\]
where each $\mu_\gamma$ is the restriction of $\Hcal^1$
to the $1$-rectifiable set $E_\gamma$.
\end{enumerate}
\end{theorem}

An Alberti representation of a Euclidean measure splits it into measures concentrated on ``fragments'' of curves. In general, these fragments cannot be glued together to obtain a \(1\)-dimensional normal current since  the  boundary may have infinite mass. Nevertheless, the ``holes'' of every curve appearing in an Alberti representation of a measure \(\nu\in\mathcal \Mcal_+(\R^d)\) can be ``filled'' in such a way as to produce a normal $1$-current \(T\) with \(\nu\ll \|T\|\). Moreover, if the representation has  directions in a cone $C$, then the constructed normal current $T$ has orienting vector \(\vec{T}\) in $C\setminus \{0\}$ almost everywhere (with respect to $\|T\|$). Indeed, we have the following lemma, which is essentially~\cite[Corollary~6.5]{AlbertiMarchese16}; it can be interpreted as a partial converse to Theorem~\ref{s-decompcurr}:

\begin{lemma} \label{lem:1curr_Arepr}
Let $\nu \in \Mcal_+(\R^d)$ be a finite Radon measure. If there is an Alberti representation $\nu = \int \nu_{\gamma} \dd \pi(\gamma)$ with directions in a cone $C$, then there exists a normal $1$-current $T \in \Nbf_1(\R^d)$ such that $\vec{T}(x) \in C\setminus \{0\}$ for $\|T\|$-almost every $x \in \R^d$ and $\nu \ll \|T\|$.
\end{lemma}

\begin{proof}
For the purpose of illustration we sketch the proof.

\proofstep{Step~1.}
Given \(\nu\) as in the statement, we claim that there exists a normal   $1$-current $T=\vec{T} \|T\|$ with \(\mathbf M(T)\le1 \) and \(\mathbf M(\partial T)\le 2\) such that $\vec{T}(x)\in C$, for $\|T\|$-almost every $x$ and  that $\nu$ is not singular with respect to $\|T\|$. 

The claim follows from the proof of~\cite[Lemma 6.12]{AlbertiMarchese16}. For the sake of completeness let us present the main line of reasoning. By arguing as in Step~1 of the proof of~\cite[Lemma 6.12]{AlbertiMarchese16},  to  every $\gamma\in\Gamma(\R^d)$ with $\gamma'(t) \in C$ and a Borel measure $\nu_\gamma\ll \Hcal^1 \restrict \im\gamma$, we can associate a \(1\)-Lipschitz map $\psi_{\nu_\gamma} \colon [0,1]\to\R^d$ satisfying 
\begin{equation*}\label{eq:propT}
\nu_\gamma(\im(\psi_{\nu_\gamma}))>0
\qquad\text{and}\qquad
\psi_{\nu_\gamma}'(t)\in C\setminus \{0\} \quad\text{for $\Hcal^1$-a.e. $t\in [0,1]$.}
\end{equation*}
This map can moreover be chosen such that \(\gamma\mapsto \psi_{\nu_\gamma}\) coincides with a Borel measurable map \(\pi\)-almost everywhere once we endow the set of curves with the topology of uniform convergence, see Step~3 in the proof of~\cite[Lemma 6.12]{AlbertiMarchese16}.

Let $T_{{\nu_\gamma}}:=\llbracket \im\psi_{\nu_\gamma},\tau_{\psi_{\nu_\gamma}},1 \rrbracket$ be the rectifiable $1$-current associated to $\psi_{\nu_\gamma}$ and  set
\[
  T := \int T_{\nu_{\gamma}} \dd \pi(\gamma) \; .
\]
Since  \(\psi_{\nu_\gamma}\) is \(1\)-Lipschitz, \(\Hcal ^1(\im\psi_{\nu_\gamma})\le 1\) and thus   $\mathbf M (T)\le 1$. Moreover, for all smooth  compactly supported functions  \(f \colon \R^d \to \R \) we have
 \[
 \langle \partial T,f \rangle= \langle T,d f\rangle =\int f(\psi_{\nu_\gamma}(1))-f(\psi_{\nu_\gamma}(0))  \dd \pi(\gamma) \; ,
 \] 
so that \(\mathbf M (\partial T)\le 2\).

By assumption, \(\vec {T}(x)\in C\setminus \{0\}\) for \(\|T\|\)-almost every \(x \in \R^d\). To show that  \(\|T\|\) and \(\nu\) are not mutually singular, for \(\pi\)-almost every \(\gamma\) set
\[
  \nu'_{\gamma}:=\nu_\gamma\restrict \im\, \psi_{\nu_\gamma}  \qquad\text{and}\qquad
  \nu':=\int \nu_{\gamma}' \dd \pi (\gamma) \; ,
\]
so that \(\nu'\ne 0\) and \(\nu'\le \nu\). We will now establish that \(\nu'\ll \|T\|\), for  which we will prove that \(\nu\) and \(\|T\|\) are not mutually singular. Let \(E \subset \R^d\) be such that \(\|T\|(E)=0\). Using
\[
  T=\int \llbracket \im\psi_{\nu_\gamma},\tau_{\psi_{\nu_\gamma}},1 \rrbracket \dd \pi (\gamma)  \qquad\text{with}\qquad
  \tau_{\psi_{\nu_\gamma}}=\frac{\psi_{\nu_\gamma}'}{|\psi_{\nu_\gamma}'|} \in C \; ,
\]
we get 
\[
\Hcal^1(\im\psi_{\nu_\gamma}\cap E)=0\qquad \text{for \(\pi\)-a.e. \(\gamma\).}
\] 
Since by definition  \(\nu_\gamma\ll \Hcal^1\restrict \im \gamma\), we have that \( \nu'_\gamma\ll \Hcal^1\restrict \im \psi_{\nu_\gamma}\). Thus, \(\nu'(E)=0\).

\medskip
\noindent
\proofstep{Step~2.}
Let us define   
\[
\mathcal T :=\setb{ T\in \mathbf N_1(\R^d) }{ \text{$\mathbf M (T)\le 1$, $\mathbf M(\partial T)\le 2$ and $\vec T\in C$ \(\|T\|\)-a.e.} }
\]
and
\[
\mathcal T_\nu :=\setb{T\in \mathcal T}{ \text{\(\nu\) and \(T\) are not singular}} \; .
\]
Note that if \(C=\setn{v\in\R^d}{v\cdot w\geq(1-\theta)\|v\|}\) for some \(w\in \mathbb{S}^{d-1}\), \(\theta \in (0,1)\), then \(\vec T\in C\) almost everywhere implies that 
\begin{equation}\label{eq:coneC}
\|T\| \ge T\cdot w\ge (1-\theta) \| T\|
\end{equation}
as measures (here we are identifying \(T\) with an \(\R^d\)-valued Radon measure and use the pointwise scalar product). Moreover, as a consequence of the  Radon--Nikod\'{y}m theorem,  for every \(T\in \mathcal T_\nu \) we may write 
\[
\nu =g_{\|T\|} \|T\|+\nu_{\|T\|}^s \qquad\text{with}\qquad
\nu_{\|T\|}^s\perp \|T\| \;,\; 
\int g_{\|T\|} \dd \|T\|>0 \; .
\]
Let us set  \(M:=\sup_{T\in \mathcal T_\nu } \int g_{\|T\|} \dd \|T\|>0\)
and let \(T_k\in \mathcal T_\nu \) be a sequence with
\[
\int g_{\|T_k\|} \dd \|T_k\|\to M.
\]
Define
\[
  T:=\sum_{k} 2^{-k} T_k
\]
and note that \(T\in \mathcal T\). Moreover, by ~\eqref{eq:coneC}, \( \|T_k\|\ll  \|T\|\) for all \(k\in \N\), so that there exist \(h_k \colon \R^d \to \R \) with
\[
\int_E h_k \dd \|T\|=\int_E g_{\|T_k\|} \dd \|T_k\| \le \nu (E) \qquad\textrm{for all Borel sets \(E \subset \R^d\).}
\]
In particular, \(T\in \mathcal T_\nu\) and \(h_k\le g_{\|T\|}\). Set  \(m_k=\max_{1\le j\le k} h_j\). By the monotone convergence theorem, \(m_k\to m_\infty\le g_{\|T\|} \) in \(\Lrm^1(\R^d,\|T\|)\) and 
\[
M\le \lim_{k\to\infty} \int m_k \dd \|T\| = \int m_\infty  \dd \|T\|\le \int g_{\|T\|}\dd \|T\|\le M.
\] 
Hence, \(M\) is actually a maximum and it is attained by \(T\).

We now claim that \(\nu \ll \|T\|\). Indeed, assume by contradiction that \(\nu=g_{\|T\|}\dd \|T\|+\nu^s_{\|T\|}\) with \(\nu_{\|T\|}^s\ne 0\). Since the Alberti representation of \(\nu\) induces an Alberti representation of \(\nu_{\|T\|}^s\),  we can apply Step~1 to find a normal $1$-current
\[
  S\in \mathcal T_{\nu_{\|T\|}^s}\subset \mathcal T_\nu
\]
such that \(\nu_{\|T\|}^s\) and \(\|S\|\) are not mutually singular. In particular, if \(\nu =g_{\|S\|} \dd \|S\|+\nu_{\|S\|}^s\), then there exists a Borel set \(F \subset \R^d\) such that 
\begin{equation}\label{eq:cont}
  \|T\|(F)=0
  \qquad\textrm{and}\qquad
  \int_F g_{\|S\|}\dd \|S\|>0.
\end{equation}
 Let us define  \(W:=(T+S)/2\) and note that by~\eqref{eq:coneC} it holds that \(\|T\|,\|S\|\ll \|W\|\) so that \(W\in \mathcal T_\nu\). Moreover,  there are  functions \(h_{T}\,, h_S\le g_{\|W\|} \) such that 
 \[
 \int_E h_{T} \dd \|W\|=  \int_E g_{\|T\|} \dd \|T\| \;,\qquad  \int_E h_{S} \dd \|W\|=  \int_E g_{\|S\|} \dd \|S\|
 \]
 for all Borel sets \(E\). However, for \(F\) as in~\eqref{eq:cont} we obtain 
 \[
 M\ge \int_{\R^d} g_{\|W\|} \dd \|W\|\ge \int_{\R^d} g_{\|T\|} \dd \|T\|+ \int_{F } g_{\|S\|} \dd \|S\|>M,
 \]
 a contradiction.
 \end{proof}

\section{Proof of Cheeger's conjecture}

The key tool to prove Cheeger's conjecture is the  following result from~\cite[Corollary~1.12]{DePhilippisRindler16}:

\begin{theorem} \label{thm:1curr}
Let $T_1=\vec{T}_1 \|T_1\|,\ldots,T_d=\vec{T}_d \|T_d\|\in \mathbf{N}_1(\R^d)$ be $1$-dimensional normal currents. Let $\nu\in \Mcal_+(\R^d)$ be a positive Radon measure such that
\begin{enumerate}
\item[(i)]  $\nu \ll \|T_i\|$ for $i=1,\ldots,d$, and
\item[(ii)] $\spn\{\vec {T}_1(x),\ldots,\vec{T}_d(x)\}=\R^d$ for $\nu$-almost every $x$.
\end{enumerate}
Then, $\nu \ll \Lcal^{d}$.
\end{theorem}

Combining the above result with Lemma~\ref{lem:1curr_Arepr} we immediately get the following:

\begin{lemma} \label{lem:abscont}
Let $\nu \in \Mcal_+(\R^d)$ have $d$ independent Alberti representations. Then, $\nu \ll \Lcal^d$.
\end{lemma}

\begin{proof}
Denote by $C_1,\ldots,C_d$ independent cones such that there are $d$ Alberti representations having directions in these cones. By Lemma~\ref{lem:1curr_Arepr} there are $d$ normal $1$-dimensional currents $T_1=\vec{T}_1 \|T_1\|,\ldots, T_d = \vec{T}_d \|T_d\|\in \mathbf{N}_1(\R^d)$ such that
\[
  \nu \ll \|T_i\|  \qquad\text{for $i=1,\ldots,d$,}
\]
and $\vec{T}_i(x) \in C_i$ for $\nu$-almost every $x \in \R^d$. By the independence of the cones,
\[
 \spn\bigl\{\vec{T}_1(x),\ldots,\vec{T}_d(x)\bigr\}=\R^d\qquad\textrm{for $\nu$-a.e.\ $x \in \R^d$.}
\]
This implies $\nu\ll \Lcal^d$ via Theorem~\ref{thm:1curr}. 
\end{proof}

In order to use the above result to prove Theorem~\ref{thm:main} one further needs the following \enquote{push-forward lemma}.

\begin{lemma} \label{lem:Arepr_push-forward}
Let $(X,\rho,\mu)$ be a Lipschitz differentiability space with a $d$-chart $(U,\phi)$. If $\mu\restrict U$ has $d$ $\phi$-independent Alberti representations, then also the push-forward $\phi_\# (\mu\restrict U) \in \Mcal_+(\R^d)$ has $d$ independent Alberti representations.
\end{lemma}

\begin{proof}
It is enough to show that if there exists a representation of the form   $\mu\restrict U = \int \mu_\gamma \dd \pi(\gamma)$ with $\phi$-directions in a cone $C$ (i.e.\ such that \((\varphi\circ \gamma)'(t)\in C\setminus \{0\}\) for almost all \(t \in \dom \gamma\) and for \(\pi\)-almost every \(\gamma\)), then we can build an Alberti representation
\[
\phi_\#(\mu\restrict U)=\int \nu_{\bar \gamma}\dd\bar \pi(\bar \gamma)\qquad
\text{with}\qquad
 \bar\pi\in\mathcal P(\Gamma(\R^d)).
\]
with \(\bar \gamma'(t)\in C\setminus\{0\}\) for \(\bar \pi\)-almost every \(\bar \gamma\) and almost every \(t\in \dom \bar \gamma\). To this end consider the map $\Phi \colon \Gamma(X)\to \Gamma (\R^d)$ given by $\Phi(\gamma) := \phi \circ \gamma$ and let \(\bar \pi:=\Phi_\# \pi\in \Mcal_+(\Gamma(\R^d))\). Note that, by the very definition of the push-forward  measure, for \(\bar\pi\)-almost every \(\bar \gamma\),  it holds that \(\bar \gamma=\varphi\circ \gamma\) for some \(\gamma\in \Gamma (X)\).  

By considering \(\pi\) as a probability measure defined on the Polish space \(\mathcal K\) defined in~\eqref{eq:K}, and noting that $\pi$ is concentrated on \(\Gamma(X)\), we can  apply the disintegration theorem for measures~\cite[Theorem~5.3.1]{AmbrosioGigliSavare05book} to show that for \(\bar\pi\)-almost every \(\bar \gamma\), there exists a Borel probability measure   \(\eta_{\bar \gamma}\)  concentrated on 
\(\Phi^{-1}(\bar\gamma)\) and such that 
\[
\pi(A) =\int  \eta_{\bar \gamma}(A) \dd\bar \pi(\bar \gamma)\qquad\textrm{for all Borel sets \(A\subset\Gamma(X)\).}
\]
Note also that, by the  disintegration theorem, the map \(\bar \gamma \mapsto \eta_{\bar \gamma}\) is Borel measurable. Let us now set 
\[
\nu_{\bar \gamma} := \int_{ \Phi^{-1}(\bar\gamma)} \phi_\#(\mu_\gamma) \dd \eta_{\bar \gamma} (\gamma).
\]
Clearly, we have the representation
\[
\phi_\#(\mu\restrict U)=\int \nu_{\bar \gamma}\dd\bar \pi(\bar \gamma)
\]
and \(\bar \gamma'(t)=(\varphi\circ \gamma )'(t)\in C\setminus\{0\}\) for \(\bar \pi\)-almost every \(\bar \gamma\) and almost every \(t\in \dom \bar \gamma\). Hence, to conclude we only have to show that 
\[
\nu_{\bar \gamma}\ll \Hcal^1\restrict \im \,\bar \gamma\qquad \textrm{for \(\bar \pi\)-a.e. \(\bar\gamma\).}
\]
Let \(E\) be a set with \(\Hcal^1(E\cap  \im\,\bar \gamma)=0\). Since \(\bar \gamma'(t)\ne 0\) for almost every \(t\in \dom \gamma\), the area formula implies that \( \Lcal^1(\bar \gamma^{-1}(E))=0\). If \(\gamma\in \Phi^{-1}(\bar \gamma)\), say \(\bar \gamma=\varphi\circ \gamma\), then 
\[
\Hcal^1(\varphi^{-1}(E)\cap \im\,\gamma)\le \Hcal^1(\gamma(\bar \gamma^{-1}(E)))=0\qquad \textrm{for all \(\gamma\in \Phi^{-1}(\bar \gamma)\).}
\]
Hence, \(\mu_\gamma (\varphi^{-1}(E))=0\) for all \(\gamma\in \Phi^{-1}(\bar \gamma)\), which immediately gives
\[
\nu_{\bar \gamma}(E)=\int_{\Phi^{-1}(\bar \gamma)} \mu_\gamma (\varphi^{-1}(E)) \dd \eta_{\bar \gamma} (\gamma)=0 \; .
\]
This concludes the proof.
\end{proof}

\begin{proof}[Proof of Theorem~\ref{thm:main}]
Let $(U,\phi)$ be a $d$-chart. By Theorem~\ref{thm:Bate} there are $d$ $\phi$-independent Alberti representations of $\mu \restrict U_k$, where $U = \bigcup_{k \in \N} U_k$ is the decomposition from Bate's theorem. Then, via Lemma~\ref{lem:Arepr_push-forward}, the push-forward $\phi_\# (\mu \restrict U_k)$ also has $d$ independent Alberti representations. Finally, Lemma~\ref{lem:abscont} yields $\phi_\# (\mu \restrict U_k) \ll \Lcal^d$ and this concludes the proof.
\end{proof}



\begin{thebibliography}{ACP10}

\bibitem[ACP05]{AlbertiCsornyeiPreiss05}
G.~Alberti, M.~Cs\"{o}rnyei, and D.~Preiss.
\newblock Structure of null sets in the plane and applications.
\newblock In {\em {Proceedings of the Fourth European Congress of Mathematics
  (Stockholm, 2004)}}, pages 3--22. European Mathematical Society, 2005.

\bibitem[ACP10]{AlbertiCsornyeiPreiss10}
G.~Alberti, M.~Cs\"{o}rnyei, and D.~Preiss.
\newblock Differentiability of lipschitz functions, structure of null sets, and
  other problems.
\newblock In {\em {Proceedings of the International Congress of Mathematicians
  2010 (Hyderabad 2010)}}, pages 1379--1394. European Mathematical Society,
  2010.

\bibitem[AGS05]{AmbrosioGigliSavare05book}
L.~Ambrosio, N.~Gigli, and G.~Savar\'{e}.
\newblock {\em Gradient flows in metric spaces and in the space of probability
  measures}.
\newblock Lectures in Mathematics ETH Z\"urich. Birkh\"auser, 2005.

\bibitem[Alb93]{Alberti93}
G.~Alberti.
\newblock Rank one property for derivatives of functions with bounded
  variation.
\newblock {\em Proc. Roy. Soc. Edinburgh Sect. A}, 123:239--274, 1993.

\bibitem[AM16]{AlbertiMarchese16}
G.~Alberti and A.~Marchese.
\newblock On the differentiability of lipschitz functions with respect to
  measures in the {Euclidean} space.
\newblock {\em Geom. Funct. Anal.}, 26:1--66, 2016.

\bibitem[Bat15]{Bate15}
D.~Bate.
\newblock Structure of measures in {L}ipschitz differentiability spaces.
\newblock {\em J. Amer. Math. Soc.}, 28:421--482, 2015.

\bibitem[Che99]{Cheeger99}
J.~Cheeger.
\newblock Differentiability of {L}ipschitz functions on metric measure spaces.
\newblock {\em Geom. Funct. Anal.}, 9:428--517, 1999.

\bibitem[CK06]{CheegerKleiner06}
J.~Cheeger and B.~Kleiner.
\newblock On the differentiability of {L}ipschitz maps from metric measure
  spaces to {B}anach spaces.
\newblock In {\em Inspired by {S}. {S}. {C}hern}, volume~11 of {\em Nankai
  Tracts Math.}, pages 129--152. World Scientific, 2006.

\bibitem[CK09]{CheegerKleiner09}
J.~Cheeger and B.~Kleiner.
\newblock Differentiability of {L}ipschitz maps from metric measure spaces to
  {B}anach spaces with the {R}adon-{N}ikod\'ym property.
\newblock {\em Geom. Funct. Anal.}, 19:1017--1028, 2009.

\bibitem[DR16]{DePhilippisRindler16}
G.~{De Philippis} and F.~Rindler.
\newblock On the structure of $\mathcal{A}$-free measures and applications.
\newblock {\em Ann. of Math.}, 2016.
\newblock to appear, arXiv:1601.06543.

\bibitem[Fed69]{Federer69book}
H.~Federer.
\newblock {\em Geometric measure theory}.
\newblock Die Grundlehren der mathematischen Wissenschaften, Band 153.
  Springer-Verlag New York Inc., New York, 1969.

\bibitem[Gon12]{Gong11}
J.~Gong.
\newblock Rigidity of derivations in the plane and in metric measure spaces.
\newblock {\em Illinois J. Math.}, 56:1109--1147, 2012.

\bibitem[GP16]{GigliPasqualetto16}
N.~Gigli and E.~Pasqualetto.
\newblock Behaviour of the reference measure on $\mathsf{RCD}$ spaces under
  charts.
\newblock arXiv:1607.05188, 2016.

\bibitem[Jon11]{Jones11talk}
P.~Jones.
\newblock Product formulas for measures and applications to analysis and
  geometry, 2011.
\newblock Talk given at the conference ``Geometric and algebraic structures in
  mathematics'', Stony Brook University, May 2011,
  http://www.math.sunysb.edu/Videos/dennisfest/.

\bibitem[Kei04]{Keith04}
S.~Keith.
\newblock A differentiable structure for metric measure spaces.
\newblock {\em Adv. Math.}, 183:271--315, 2004.

\bibitem[KM16]{KellMondino16}
M.~Kell and A.~Mondino.
\newblock On the volume measure of non-smooth spaces with {R}icci curvature
  bounded below.
\newblock arXiv:1607.02036, 2016.

\bibitem[Pre90]{Preiss90}
D.~Preiss.
\newblock Differentiability of lipschitz functions on banach spaces.
\newblock {\em J. Funct. Anal.}, 91:312--345, 1990.

\bibitem[PS15]{PreissSpeight14}
D.~Preiss and G.~Speight.
\newblock Differentiability of {Lipschitz} functions in {Lebesgue} null sets.
\newblock {\em Invent. Math.}, 199:517--559, 2015.

\bibitem[Sch16]{Schioppa13}
A.~Schioppa.
\newblock Derivations and {A}lberti representations.
\newblock {\em Adv. Math.}, 293:436--528, 2016.

\bibitem[Smi93]{Smirnov93}
S.~K. Smirnov.
\newblock Decomposition of solenoidal vector charges into elementary solenoids,
  and the structure of normal one-dimensional flows.
\newblock {\em Algebra i Analiz}, 5:206--238, 1993.
\newblock translation in St. Petersburg Math. J. 5 (1994), 841--867.

\bibitem[Zah46]{Zahorski46}
Z.~Zahorski.
\newblock Sur l'ensemble des points de non-d\'erivabilit\'e d'une fonction
  continue.
\newblock {\em Bull. Soc. Math. France}, 74:147--178, 1946.

\end{thebibliography}

\end{document}